\documentclass[12pt]{article}
\usepackage{amsmath,amssymb,amsbsy,amsfonts,amsthm,latexsym,
            amsopn,amstext,amsxtra,euscript,amscd,amsthm,url}

\usepackage[usenames, dvipsnames]{color}
\usepackage[utf8]{inputenc}
\usepackage[english]{babel}
\usepackage{mathrsfs}
\usepackage{caption}

\newtheorem{Thm}{Theorem}[section]
\newtheorem{lem}[Thm]{Lemma}

\DeclareMathOperator{\Reg}{Reg}

\newcommand{\Ocal}{O}

\DeclareMathOperator{\id}{id}
\global\long\def\epsilon{\varepsilon}

\begin{document}
\date{}

\title{Weighted averages of Bernoulli polynomials and the Gamma function over regular integers modulo $n$}
\author{Waseem Alass and Sumaia Saad Eddin}

\maketitle
{\def\thefootnote{}
\footnote{{\it Mathematics Subject Classification 2010: 11A25, 11N37.}} 

\begin{abstract}
We investigate weighted averages involving Bernoulli polynomials and the Gamma function over regular integers modulo n. Building upon recent identities of Kiuchi and Matsuoka together with our previous transformation formulas for weighted averages over regular integers, we derive explicit asymptotic formulas for a broad class of arithmetic functions. As applications, asymptotic formulas are obtained for the identity, Möbius, divisor and Jordan totient functions.
\end{abstract}

\maketitle

\section{Introduction and results}

Arithmetic functions involving greatest common divisors have long played an important role in analytic and multiplicative number theory. One of the best-known examples is the $\gcd$-sum function (also known as Pillai's arithmetical function),
\[
P(n)=\sum_{k=1}^{n}\gcd(k,n),
\]
which was introduced by Pillai~\cite{Pillai} in 1933. Its analogue over regular integers modulo $n$ was later introduced by T\'oth~\cite{T2},
\[
\widetilde{P}(n)=\sum_{k\in\Reg_n}\gcd(k,n).
\]

Recall that an integer $k$ is called \emph{regular modulo $n$} if there exists an integer $x$ satisfying
\[
k^2x\equiv k\pmod n.
\]
Equivalently, $k$ is regular modulo $n$ if and only if $\gcd(k,n)$ is a unitary divisor of $n$. Here a divisor $d$ of $n$ is called \emph{unitary}, denoted by $d\mid\mid n$, if
\[
d\mid n
\qquad\text{and}\qquad
\gcd\!\left(d,\frac{n}{d}\right)=1.
\]
We denote by
\[
\Reg_n=\{k\in\mathbb N:1\le k\le n,\;k \text{ is regular }\pmod n\}
\]
the set of regular integers modulo $n$.

The arithmetic and asymptotic properties of these functions have attracted considerable attention over the past decades. For a comprehensive survey on the $\gcd$-sum function and its generalizations, we refer the reader to~\cite{T1}.

Throughout the paper, $\mu$, $\phi$, $\tau$ and $\Lambda$ denote the M\"obius function, Euler's totient function, the divisor function and the von Mangoldt function, respectively. For every integer $s$, the Jordan totient function is defined by
\[
\phi_s(n)
=
\sum_{d\mid n}
d^s\mu\!\left(\frac nd\right)
=
n^s\prod_{p\mid n}
\left(1-\frac1{p^s}\right).
\]
In particular, $\phi_1=\phi$.

In a recent paper~\cite{WA}, we established general transformation formulas for weighted averages of arithmetic functions over regular integers modulo $n$. These formulas provide a unified framework for studying the asymptotic behaviour of a broad class of weighted averages. As applications, explicit asymptotic formulas were obtained for several classical multiplicative functions, including the identity, M\"obius, divisor and Jordan totient functions.

On the other hand, Apostol and T\'oth~\cite{A.T} established remarkable identities involving Bernoulli polynomials, the Gamma function and regular integers modulo $n$. More recently, Kiuchi and Matsuoka~\cite{K.M} generalized these identities by proving that, for every arithmetic function $f$, every positive integer $m\ge1$, and every fixed positive integer $r$,
\begin{equation}
\label{eq3}
\sum_{n\le x}\frac1n
\sum_{k\in\Reg_n}
f(\gcd(k,n))
B_{2m}\!\left(\frac{k}{n}\right)
=
B_{2m}
\sum_{\substack{d\ell\le x\\(d,\ell)=1}}
\frac{f(\ell)}{\ell}
\frac{\phi_{1-2m}(d)}{d},
\end{equation}
together with the identity
\begin{multline}
\label{eq5}
\sum_{n\le x}\frac1n
\sum_{k\in\Reg_n}
f(\gcd(k,n))
\log\Gamma\!\left(\frac{k}{n}\right)
=
\log\sqrt{2\pi}
\left(
\sum_{\substack{d\ell\le x\\(d,\ell)=1}}
\frac{f(\ell)}{\ell}
\frac{\phi(d)}{d}
-
\sum_{d\ell\le x}
\frac{f(d\ell)}{d\ell}
\right)
\\
-\frac12
\sum_{\substack{d\ell\le x\\(d,\ell)=1}}
\frac{f(\ell)}{\ell}
\frac{\Lambda(d)}{d}.
\end{multline}
If, in addition, $f(n)\neq0$ for every $n\in\mathbb N$, then
\begin{equation}
\label{eq4}
\sum_{n\leq x}\frac{1}{f(n)}
\sum_{k\in\Reg_n}
f(\gcd(k,n))
B_{2m}\!\left(\frac{k}{n}\right)
=
B_{2m}
\sum_{\substack{d\ell\leq x\\(d,\ell)=1}}
\frac{f(\ell)\phi_{1-2m}(d)}
{f(d\ell)}.
\end{equation}
Moreover,
\begin{multline}
\label{eq6}
\sum_{n\leq x}\frac{1}{f(n)}
\sum_{k\in\Reg_n}
f(\gcd(k,n))
\log\Gamma\!\left(\frac{k}{n}\right)
=
\log\sqrt{2\pi}
\left(
\sum_{\substack{d\ell\leq x\\(d,\ell)=1}}
\frac{f(\ell)\phi(d)}
{f(d\ell)}
-
x+\theta(x)+\frac12
\right)
\\
-
\frac12
\sum_{\substack{d\ell\leq x\\(d,\ell)=1}}
\frac{f(\ell)\Lambda(d)}
{f(d\ell)},
\end{multline}
where
\[
\theta(x)=x-\lfloor x\rfloor-\frac12.
\]
Although identities~\eqref{eq3} and~\eqref{eq5} are quite general and elegant, their asymptotic consequences have not yet been systematically investigated. The purpose of the present paper is to fill this gap. By combining the transformation formulas established in~\cite{WA} with the identities of Kiuchi and Matsuoka, we derive explicit asymptotic formulas for weighted averages involving Bernoulli polynomials and the logarithm of the Gamma function over regular integers modulo $n$.

As applications, we obtain asymptotic formulas corresponding to several classical multiplicative functions, namely the identity function $\id$, the M\"obius function $\mu$, the divisor function $\tau$, and the Jordan totient functions $\phi$ and $\phi_2$. The resulting formulas exhibit explicit Euler product representations for the main terms together with explicit error estimates.

Most of the auxiliary estimates required in the proofs were established in our previous paper~\cite{WA}. For the convenience of the reader, these results are recalled in Section~2 before we prove the main theorems. Throughout the paper, $p$ always denotes a prime number, $\zeta(s)$ denotes the Riemann zeta function, and $\zeta'(s)$ denotes its derivative with respect to $s$.
\begin{Thm}
\label{Thm5}
For any sufficiently large positive number $x>2$ and positive integer $m\geq 1$, we have 
\begin{multline}
\label{eqthm21}
\sum_{n\leq x}\frac{1}{n}\sum_{k\in \Reg_n}\gcd(k,n)B_{2m}\left( \frac{k}{n}\right)=\\ \prod_p \left( 1-\frac{(p-1)(p^{2m-1}-1)}{p(p^{2m+1}-1)}\right)B_{2m}x+\Ocal_m\left( (\log x)^2\right),
\end{multline}
\begin{multline}
\label{eqthm22}
\sum_{n\leq x}\frac{1}{n}\sum_{k\in \Reg_n}\phi\left(\gcd(k,n)\right)B_{2m}\left( \frac{k}{n}\right)=\\ B_{2m}\prod_{p}\left(1-\frac{p(p^{2m-1}-1)}{(p+1)(p^{2m}-1)}\right)\frac{x}{\zeta(2)}+\Ocal_m\left((\log x)^3\right), 
\end{multline}
\begin{equation}
\label{eqthm23}
\sum_{n\leq x}\frac{1}{n}\sum_{k\in \Reg_n}\tau\left(\gcd(k,n)\right)B_{2m}\left( \frac{k}{n}\right)=\Ocal_m\left((\log x)^5\right),
\end{equation}
\begin{equation}
\label{eqthm24}
\sum_{n\leq x}\frac{1}{n}\sum_{k\in \Reg_n}\mu\left(\gcd(k,n)\right)B_{2m}\left( \frac{k}{n}\right)=\Ocal_m\left((\log x)^3 \right).
\end{equation}
Furthermore, we have 
\begin{equation}
\label{eqthm25}
\sum_{n\leq x}\frac{1}{\phi(n)}\sum_{k\in \Reg_n}\phi\left(\gcd(k,n)\right)B_{2m}\left( \frac{k}{n}\right)=B_{2m}\frac{\zeta(2m+1)}{\zeta(2)}x +\Ocal_m\left( \left(\log x\right)^4\right),
\end{equation}
\begin{multline}
\label{eqthm26}
\sum_{n\leq x}\frac{1}{\phi_2(n)}\sum_{k\in \Reg_n}\phi_2\left(\gcd(k,n)\right)B_{2m}\left( \frac{k}{n}\right)=\\B_{2m}\prod_p\left(1-\frac{p(p^{2m-1}-1)}{(p+1)(p^{2m+2}-1)}\right)x+\Ocal_m\left(\frac{(\log x)^3}{x}\right).
\end{multline}
\end{Thm}
\begin{Thm}
\label{Thm6}
For  any sufficiently large positive number $x>2$, we have
\begin{equation}
\label{eqthm61}
\sum_{n\leq x}\frac{1}{n}\sum_{k\in \Reg_n}\gcd(k,n)\log \Gamma\left( \frac{k}{n}\right)=
\frac{\log \sqrt{2\pi}}{\zeta(2)}\prod_{p}\left( 1-\frac{1}{p(p+1)}\right)x\log x+\Ocal\left( x\right),
\end{equation}
\begin{multline}
\label{eqthm62}
\sum_{n\leq x}\frac{1}{n}\sum_{k\in \Reg_n}\tau \left(\gcd(k,n)\right)\log \Gamma\left( \frac{k}{n}\right)=\\
\log \sqrt{2\pi}\frac{1}{\zeta(2)}\prod_{p}\left(1+\frac{(2p^2-1)p}{(p-1)^2(p+1)^3}\right)x+\Ocal\left((\log x)^5 \right),
\end{multline}
\begin{multline}
\label{eqthm63}
\sum_{n\leq x}\frac{1}{n}\sum_{k\in \Reg_n}\mu \left(\gcd(k,n)\right)\log \Gamma\left( \frac{k}{n}\right)=\\ \log \sqrt{2\pi}\frac{1}{\zeta(2)}\prod_p\left(1-\frac{1}{p(p+1)}\right)x +\Ocal\left((\log x)^3\right),
\end{multline}
Furthermore, we have 
\begin{multline}
\label{eqthm64}
\sum_{n\leq x}\frac{1}{\phi(n)}\sum_{k\in \Reg_n}\phi \left(\gcd(k,n)\right)\log \Gamma\left( \frac{k}{n}\right)=
\frac{\log \sqrt{2\pi}}{\zeta(2)}x\log x \\+\frac{\log \sqrt{2\pi}}{\zeta(2)}\left(2\gamma-1+\frac{\zeta'(2)}{\zeta(2)}+\frac{\zeta'(2)}{2\log \sqrt{2\pi}}-\zeta(2)\right)x\\+\Ocal\left( x^{1/2}\exp\left(-A(\log x)^{3/5}(\log \log x)^{-1/5}\right)\right), 
\end{multline}
\begin{multline}
\label{eqthm65}
\sum_{n\leq x}\frac{1}{\phi_2(n)}\sum_{k\in \Reg_n}\phi_2 \left(\gcd(k,n)\right)\log \Gamma\left( \frac{k}{n}\right)=\\
\log \sqrt{2\pi}\left(\prod_{p}\left( 1+\frac{1}{(p+1)^2}\right)-\frac{K}{2\log \sqrt{2\pi}}-1 \right)x+\Ocal\left((\log x)^4\right),
\end{multline}
where $A>0$ and $$K=\sum_{d\geq 1}\frac{\Lambda(d)}{d^2\psi(d)}.$$
\end{Thm}

\section{Auxiliary results}

In this section, we collect the auxiliary estimates required in the
proofs of Theorems~\ref{Thm5} and~\ref{Thm6}. The first four lemmas were
established in our previous work~\cite{WA}. We recall them here for the
convenience of the reader.

\begin{lem}
\label{Ex1}
Let $x\geq2$. Then
\begin{equation}
\label{Ex11}
\sum_{\substack{d\ell\leq x\\(d,\ell)=1}}
\frac{\phi(d)}{d}
=
\frac{x}{\zeta(2)}
\left(
K_1\left(\log x-\frac12\right)+K_2
\right)
+
\Ocal\!\left(
x^{1/2}
\exp\!\left(
-C\frac{(\log x)^{3/5}}
{(\log\log x)^{1/5}}
\right)
\right),
\end{equation}
where $C>0$,
\begin{equation*}
\label{eq0000}
K_1
=
\prod_p
\left(
1-\frac{1}{p(p+1)}
\right),
\end{equation*}
and
\begin{equation*}
\label{eq00000}
K_2
=
K_1
\left(
2\gamma-\frac12
-
2\frac{\zeta'(2)}{\zeta(2)}
\right)
-
\sum_{n\geq1}
\frac{
\mu(n)\bigl(\log n-\alpha(n)+2\beta(n)\bigr)
}
{n\psi(n)}.
\end{equation*}
Here, $\gamma$ denotes Euler's constant,
\[
\alpha(n)
=
\sum_{p\mid n}
\frac{\log p}{p-1},
\qquad
\beta(n)
=
\sum_{p\mid n}
\frac{\log p}{p^2-1},
\]
and
\[
\psi(n)
=
n\prod_{p\mid n}\left(1+\frac1p\right)
\]
is the Dedekind function.

Moreover, for every integer $m\geq1$,
\begin{equation}
\label{Ex12}
\sum_{\substack{d\ell\leq x\\(d,\ell)=1}}
\frac{\phi_{1-2m}(d)}{d}
=
C_mx+\Ocal_m\!\left((\log x)^2\right),
\end{equation}
where
\[
C_m
=
\prod_p
\left(
1-
\frac{(p-1)(p^{2m-1}-1)}
{p(p^{2m+1}-1)}
\right).
\]
\end{lem}


\begin{lem}
\label{Ex2}
Let $x\geq2$. Then
\begin{equation}
\label{Ex21}
\sum_{\substack{d\ell\leq x\\(d,\ell)=1}}
\frac{\phi(d)}{d}
\frac{\tau(\ell)}{\ell}
=
\frac{x}{\zeta(2)}
\prod_p
\left(
1+
\frac{p(2p^2-1)}
{(p-1)^2(p+1)^3}
\right)
+
\Ocal\!\left((\log x)^5\right).
\end{equation}
Furthermore, for every integer $m\geq1$,
\begin{equation}
\label{Ex22}
\sum_{\substack{d\ell\leq x\\(d,\ell)=1}}
\frac{\phi_{1-2m}(d)}{d}
\frac{\tau(\ell)}{\ell}
=
\Ocal_m\!\left((\log x)^5\right).
\end{equation}
\end{lem}


\begin{lem}
\label{Ex3}
Let $x\geq2$. Then
\begin{equation}
\label{Ex31}
\sum_{\substack{d\ell\leq x\\(d,\ell)=1}}
\frac{\phi(d)}{d}
\frac{\mu(\ell)}{\ell}
=
\frac{x}{\zeta(2)}
\prod_p
\left(
1-\frac{1}{p(p+1)}
\right)
+
\Ocal\!\left((\log x)^3\right),
\end{equation}
and
\begin{equation}
\label{Ex32}
\sum_{\substack{d\ell\leq x\\(d,\ell)=1}}
\frac{\phi_{-1}(d)}{d}
\frac{\mu(\ell)}{\ell}
=
\Ocal\!\left((\log x)^3\right).
\end{equation}
For every integer $m\geq2$,
\begin{equation}
\label{Ex33}
\sum_{\substack{d\ell\leq x\\(d,\ell)=1}}
\frac{\phi_{1-2m}(d)}{d}
\frac{\mu(\ell)}{\ell}
=
\Ocal_m\!\left((\log x)^2\right).
\end{equation}
\end{lem}


\begin{lem}
\label{Ex4}
Let $x\geq2$. Then
\begin{equation}
\label{Ex41}
\sum_{\substack{d\ell\leq x\\(d,\ell)=1}}
\frac{\phi(d)}{\phi_2(d)}
=
x
\prod_p
\left(
1+\frac{1}{(p+1)^2}
\right)
+
\Ocal\!\left((\log x)^4\right).
\end{equation}
Furthermore, for every integer $m\geq1$,
\begin{equation}
\label{Ex42}
\sum_{\substack{d\ell\leq x\\(d,\ell)=1}}
\frac{\phi_{1-2m}(d)}{\phi_2(d)}
=
x
\prod_p
\left(
1-
\frac{p(p^{2m-1}-1)}
{(p+1)(p^{2m+2}-1)}
\right)
+
\Ocal_m\left(\frac{(\log x)^3}{x}\right).
\end{equation}
\end{lem}


We shall also use the following standard estimates. Throughout, let
\[
\theta(x)=x-\lfloor x\rfloor-\frac12,
\]
and let $\tau^\star(n)$ denote the number of square-free divisors of $n$.

\begin{lem}
\label{lem1}
Let $t>1$ be an integer and let $x\geq2$. Then
\begin{equation}
\label{eqlem11}
\sum_{\substack{n\leq x\\(n,t)=1}}1
=
\frac{\phi(t)}{t}x
-
\sum_{d\mid t}
\mu(d)
\theta\!\left(\frac{x}{d}\right).
\end{equation}
Moreover,
\begin{equation}
\label{eqlem14}
\sum_{\substack{n\leq x\\(n,t)=1}}
\frac{\phi(n)}{n}
=
\frac{t\phi(t)}
{\zeta(2)\phi_2(t)}x
+
\Ocal\!\left(
\tau^\star(t)\log x
\right).
\end{equation}
\end{lem}

\begin{proof}
Identity~\eqref{eqlem11} is given in
\cite[Lemma~2.1]{K.M}. Formula~\eqref{eqlem14} may be found in
\cite{Sur} or \cite[Chapter~1, Section~I.24]{SMC}.
\end{proof}


\begin{lem}
\label{lem2}
Let $x\geq2$. There exists a positive constant $B$ such that
\begin{equation}
\label{eqlem23}
\sum_{n\leq x}\frac{\mu(n)}{n}
=
\Ocal\!\left(
\exp\!\left(
-B(\log x)^{3/5}
(\log\log x)^{-1/5}
\right)
\right).
\end{equation}
Furthermore,
\begin{equation}
\label{eqlem24}
\sum_{n\leq x}\frac{\tau(n)}{n}
=
\frac12(\log x)^2
+
2\gamma\log x
+
\Ocal(1),
\end{equation}
\begin{multline}
\label{taustar}
\sum_{n\leq x}\frac{\tau^\star(n)}{n}
=
\frac{(\log x)^2}{2\zeta(2)}
+
\frac{1}{\zeta(2)}
\left(
2\gamma+\frac{\zeta'(2)}{\zeta(2)}
\right)\log x +\left(2\gamma-1+\frac{\zeta'(2)}{\zeta(2)}\right)\frac{1}{\zeta(2)}
\\
+
\Ocal\!\left(
x^{-1/2}
\exp\!\left(
A(\log x)^{3/5}
(\log\log x)^{-1/5}
\right)
\right),
\end{multline}
where $A>0$. We also have
\begin{equation}
\label{WS3}
\sum_{n\leq x}\frac{\tau^2(n)}{n}
\ll
(\log x)^4,
\end{equation}
and
\begin{equation}
\label{WS1}
\sum_{\ell>x}\frac{\tau(\ell)}{\ell^2}
\ll
\frac{\log x}{x}.
\end{equation}
\end{lem}

\begin{proof}
Estimate~\eqref{eqlem23} is due to Jia~\cite{Jia}; see also
\cite[Lemma~2.1]{K2}.

To obtain~\eqref{eqlem24}, we use partial summation together with
\begin{equation}
\label{eqlem27}
\sum_{n\leq x}\tau(n)
=
x\log x+(2\gamma-1)x+\Delta(x),
\end{equation}
where Huxley~\cite{H} proved that
\[
\Delta(x)
=
\Ocal\!\left(
x^{131/416}
(\log x)^{26947/8320}
\right).
\]
Partial summation yields
\[
\sum_{n\leq x}\frac{\tau(n)}{n}
=
\frac12(\log x)^2+2\gamma\log x+\Ocal(1).
\]

Similarly,~\eqref{taustar} follows by partial summation from
\begin{multline}
\label{star}
\sum_{n\leq x}\tau^\star(n)
=
\frac{x}{\zeta(2)}
\left(
\log x+2\gamma-1+\frac{\zeta'(2)}{\zeta(2)}
\right)
\\
+
\Ocal\!\left(
x^{1/2}
\exp\!\left(
-A_1(\log x)^{3/5}
(\log\log x)^{-1/5}
\right)
\right),
\end{multline}
where $A_1>0$. The estimate~\eqref{WS3} follows from the classical bound
\[
\sum_{n\leq x}\tau^2(n)
\ll
x(\log x)^3;
\]
see \cite{Sur1,W1} or
\cite[Chapter~2, Section~II.13]{SMC}.

Finally, if
\[
T(y)=\sum_{n\leq y}\tau(n),
\]
then~\eqref{eqlem27} implies $T(y)\ll y\log y$. Partial summation gives
\[
\sum_{\ell>x}\frac{\tau(\ell)}{\ell^2}
=
-\frac{T(x)}{x^2}
+
2\int_x^\infty\frac{T(t)}{t^3}\,dt
\ll
\frac{\log x}{x},
\]
which proves~\eqref{WS1}.
\end{proof}

\section{Proofs}

\subsection{Proof of Theorem~\ref{Thm5}}

We first prove~\eqref{eqthm21}. Taking $f=\id$ in~\eqref{eq3}, we obtain
\[
\sum_{n\leq x}\frac{1}{n}
\sum_{k\in\Reg_n}
\gcd(k,n)
B_{2m}\!\left(\frac{k}{n}\right)
=
B_{2m}
\sum_{\substack{d\ell\leq x\\(d,\ell)=1}}
\frac{\phi_{1-2m}(d)}{d}.
\]
Applying~\eqref{Ex12}, we obtain
\[
\sum_{n\leq x}\frac{1}{n}
\sum_{k\in\Reg_n}
\gcd(k,n)
B_{2m}\!\left(\frac{k}{n}\right)
=
B_{2m}C_mx
+
\Ocal_m\!\left((\log x)^2\right),
\]
where
\[
C_m
=
\prod_p
\left(
1-
\frac{(p-1)(p^{2m-1}-1)}
{p(p^{2m+1}-1)}
\right).
\]
This proves~\eqref{eqthm21}.\\

Next, we prove~\eqref{eqthm22}. Taking $f=\phi$ in~\eqref{eq3} gives
\begin{equation}
\label{phi-Bernoulli-start}
\sum_{n\leq x}\frac{1}{n}
\sum_{k\in\Reg_n}
\phi\!\left(\gcd(k,n)\right)
B_{2m}\!\left(\frac{k}{n}\right)
=
B_{2m}
\sum_{\substack{d\ell\leq x\\(d,\ell)=1}}
\frac{\phi_{1-2m}(d)}{d}
\frac{\phi(\ell)}{\ell}.
\end{equation}
Reversing the order of summation, we write
\[
\sum_{\substack{d\ell\leq x\\(d,\ell)=1}}
\frac{\phi_{1-2m}(d)}{d}
\frac{\phi(\ell)}{\ell}
=
\sum_{d\leq x}
\frac{\phi_{1-2m}(d)}{d}
\sum_{\substack{\ell\leq x/d\\(\ell,d)=1}}
\frac{\phi(\ell)}{\ell}.
\]
Applying~\eqref{eqlem14} with $t=d$ and with $x$ replaced by $x/d$, we obtain
\begin{equation}
\label{phi-Bernoulli-reduction}
\sum_{\substack{d\ell\leq x\\(d,\ell)=1}}
\frac{\phi_{1-2m}(d)}{d}
\frac{\phi(\ell)}{\ell}
=
\frac{x}{\zeta(2)}
\sum_{d\leq x}
\frac{\phi_{1-2m}(d)\phi(d)}
{d\phi_2(d)}
+
\Ocal_m\!\left(
\log x
\sum_{d\leq x}
\frac{\tau^\star(d)}{d}
\right).
\end{equation}
Here we used $\left|\phi_{1-2m}(d)\right|\leq 1,$ for $m\geq1.$ By~\eqref{taustar}, the error term in
\eqref{phi-Bernoulli-reduction} is
$
\Ocal_m\!\left((\log x)^3\right).
$
It remains to evaluate the main sum. The arithmetic function
\[
\frac{\phi_{1-2m}(d)\phi(d)}
{d\phi_2(d)}
\]
is multiplicative, and its corresponding series converges absolutely.
A calculation of the local factors gives
\begin{equation*}
\label{phi-Bernoulli-Euler}
\sum_{d\geq1}
\frac{\phi_{1-2m}(d)\phi(d)}
{d\phi_2(d)}
=
\prod_p
\left(
1-
\frac{p(p^{2m-1}-1)}
{(p+1)(p^{2m}-1)}
\right).
\end{equation*}
Moreover, using
\[
\phi(d)\leq d,
\qquad
\phi_2(d)\geq\frac{d^2}{\tau(d)},
\qquad
\left|\phi_{1-2m}(d)\right|\leq1,
\]
we obtain
\[
\sum_{d>x}
\left|
\frac{\phi_{1-2m}(d)\phi(d)}
{d\phi_2(d)}
\right|
\ll
\sum_{d>x}\frac{\tau(d)}{d^2}
\ll
\frac{\log x}{x}
\]
by~\eqref{WS1}. Consequently,
\begin{equation}
\label{phi-Bernoulli-truncated}
\sum_{d\leq x}
\frac{\phi_{1-2m}(d)\phi(d)}
{d\phi_2(d)}=
\prod_p
\left(
1-
\frac{p(p^{2m-1}-1)}
{(p+1)(p^{2m}-1)}
\right)
+
\Ocal_m\!\left(\frac{\log x}{x}\right).
\end{equation}

Substituting~\eqref{phi-Bernoulli-truncated} into
\eqref{phi-Bernoulli-reduction}, and then using
\eqref{phi-Bernoulli-start}, yields
\begin{equation*}
\sum_{n\leq x}\frac{1}{n}
\sum_{k\in\Reg_n}
\phi\!\left(\gcd(k,n)\right)
B_{2m}\!\left(\frac{k}{n}\right)
=
\frac{B_{2m}x}{\zeta(2)}
\prod_p
\left(
1-
\frac{p(p^{2m-1}-1)}
{(p+1)(p^{2m}-1)}
\right)
+
\Ocal_m\!\left((\log x)^3\right).
\end{equation*}
This proves~\eqref{eqthm22}.\\

We next prove~\eqref{eqthm23}. Taking $f=\tau$ in~\eqref{eq3}, we obtain
\[
\sum_{n\le x}\frac1n
\sum_{k\in\Reg_n}
\tau(\gcd(k,n))
B_{2m}\!\left(\frac{k}{n}\right)
=
B_{2m}
\sum_{\substack{d\ell\le x\\(d,\ell)=1}}
\frac{\tau(\ell)}{\ell}
\frac{\phi_{1-2m}(d)}{d}.
\]
Applying~\eqref{Ex22} immediately yields
\[
\sum_{n\le x}\frac1n
\sum_{k\in\Reg_n}
\tau(\gcd(k,n))
B_{2m}\!\left(\frac{k}{n}\right)
=
\Ocal_m\!\left((\log x)^5\right),
\]
which proves~\eqref{eqthm23}.\\

We now establish~\eqref{eqthm24}. Taking $f=\mu$ in~\eqref{eq3}, we obtain
\[
\sum_{n\le x}\frac1n
\sum_{k\in\Reg_n}
\mu(\gcd(k,n))
B_{2m}\!\left(\frac{k}{n}\right)
=
B_{2m}
\sum_{\substack{d\ell\le x\\(d,\ell)=1}}
\frac{\mu(\ell)}{\ell}
\frac{\phi_{1-2m}(d)}{d}.
\]
For $m=1$, the desired estimate follows from~\eqref{Ex32}, whereas for
$m\ge2$ it follows from~\eqref{Ex33}. Consequently,
\[
\sum_{n\le x}\frac1n
\sum_{k\in\Reg_n}
\mu(\gcd(k,n))
B_{2m}\!\left(\frac{k}{n}\right)
=
\Ocal_m\!\left((\log x)^3\right),
\]
thereby proving~\eqref{eqthm24}.\\

We next turn to~\eqref{eqthm25}. Taking $f=\phi$ in~\eqref{eq4}, we obtain
\[
\sum_{n\le x}\frac1{\phi(n)}
\sum_{k\in\Reg_n}
\phi(\gcd(k,n))
B_{2m}\!\left(\frac{k}{n}\right)
=
B_{2m}
\sum_{\substack{d\ell\le x\\(d,\ell)=1}}
\frac{\phi_{1-2m}(d)}{\phi(d)}.
\]
By \cite[Corollary~3.2]{K.M},
\[
\sum_{\substack{d\ell\le x\\(d,\ell)=1}}
\frac{\phi_{1-2m}(d)}{\phi(d)}
=
\frac{\zeta(2m+1)}{\zeta(2)}x
+
\Ocal\!\left(
\sum_{d\le x}
\phi_{1-2m}(d)
\frac{\tau(d)}{\phi(d)}
\right).
\]
Kiuchi and Matsuoka estimated the error term by $x^\varepsilon$.
Using the inequality
\[
\phi(d)\tau(d)\ge d,
\]
together with~\eqref{WS3} and the bound
\[
|\phi_{1-2m}(d)|\le1,
\]
we obtain the sharper estimate
\[
\sum_{\substack{d\ell\le x\\(d,\ell)=1}}
\frac{\phi_{1-2m}(d)}{\phi(d)}
=
\frac{\zeta(2m+1)}{\zeta(2)}x
+
\Ocal_m\!\left((\log x)^4\right),
\]
which proves~\eqref{eqthm25}.\\

Finally, we prove~\eqref{eqthm26}. Taking $f=\phi_2$ in~\eqref{eq4}, we obtain
\[
\sum_{n\le x}\frac1{\phi_2(n)}
\sum_{k\in\Reg_n}
\phi_2(\gcd(k,n))
B_{2m}\!\left(\frac{k}{n}\right)
=
B_{2m}
\sum_{\substack{d\ell\le x\\(d,\ell)=1}}
\frac{\phi_{1-2m}(d)}{\phi_2(d)}.
\]
Applying~\eqref{Ex42}, we conclude that
\[
\sum_{n\le x}\frac1{\phi_2(n)}
\sum_{k\in\Reg_n}
\phi_2(\gcd(k,n))
B_{2m}\!\left(\frac{k}{n}\right)
=
B_{2m}
\prod_p
\left(
1-
\frac{p(p^{2m-1}-1)}
{(p+1)(p^{2m+2}-1)}
\right)
x
+
\Ocal_m\left(\frac{(\log x)^3}{x}\right),
\]
which completes the proof of~\eqref{eqthm26}.
\subsection{Proof of Theorem~\ref{Thm6}}

We first prove~\eqref{eqthm61}. Taking $f=\id$ in~\eqref{eq5}, we obtain
\begin{equation}
\label{eq01010}
\sum_{n\leq x}\frac{1}{n}
\sum_{k\in\Reg_n}
\gcd(k,n)
\log\Gamma\!\left(\frac{k}{n}\right)
=
\log\sqrt{2\pi}
\left(
\sum_{\substack{d\ell\leq x\\(d,\ell)=1}}
\frac{\phi(d)}{d}
-
\sum_{n\leq x}1
\right)
-
\frac12
\sum_{\substack{d\ell\leq x\\(d,\ell)=1}}
\frac{\Lambda(d)}{d}.
\end{equation}

Since $\Lambda(d)\geq0$, removing the coprimality condition gives
\[
0
\leq
\sum_{\substack{d\ell\leq x\\(d,\ell)=1}}
\frac{\Lambda(d)}{d}
\leq
\sum_{d\ell\leq x}\frac{\Lambda(d)}{d}.
\]
Moreover,
\[
\sum_{d\ell\leq x}\frac{\Lambda(d)}{d}
=
\sum_{d\leq x}
\frac{\Lambda(d)}{d}
\left\lfloor\frac{x}{d}\right\rfloor
=
-\frac{\zeta'(2)}{\zeta(2)}x
+
\Ocal(\log x),
\]
and hence
\begin{equation}
\label{Lambda-id-bound}
\sum_{\substack{d\ell\leq x\\(d,\ell)=1}}
\frac{\Lambda(d)}{d}
=
\Ocal(x).
\end{equation}

Using~\eqref{Ex11} and
\[
\sum_{n\leq x}1
=
x-\theta(x)-\frac12,
\]
in~\eqref{eq01010}, together with~\eqref{Lambda-id-bound}, we obtain
\[
\sum_{n\leq x}\frac{1}{n}
\sum_{k\in\Reg_n}
\gcd(k,n)
\log\Gamma\!\left(\frac{k}{n}\right)
=
\frac{K_1\log\sqrt{2\pi}}{\zeta(2)}
x\log x
+
\Ocal(x).
\]
This proves~\eqref{eqthm61}.\\

We next prove~\eqref{eqthm62}. Taking $f=\tau$ in~\eqref{eq5}, we have
\begin{multline}
\label{Gamma-tau-start}
\sum_{n\leq x}\frac{1}{n}
\sum_{k\in\Reg_n}
\tau(\gcd(k,n))
\log\Gamma\!\left(\frac{k}{n}\right)
=
\log\sqrt{2\pi}
\left(
\sum_{\substack{d\ell\leq x\\(d,\ell)=1}}
\frac{\phi(d)}{d}
\frac{\tau(\ell)}{\ell}
-
\sum_{n\leq x}\frac{\tau(n)}{n}
\right)
\\
-
\frac12
\sum_{\substack{d\ell\leq x\\(d,\ell)=1}}
\frac{\Lambda(d)}{d}
\frac{\tau(\ell)}{\ell}.
\end{multline}

For the final sum, positivity allows us to remove the coprimality
condition:
\begin{align*}
\sum_{\substack{d\ell\leq x\\(d,\ell)=1}}
\frac{\Lambda(d)}{d}
\frac{\tau(\ell)}{\ell}
&\leq
\sum_{\ell\leq x}
\frac{\tau(\ell)}{\ell}
\sum_{d\leq x/\ell}
\frac{\Lambda(d)}{d}.
\end{align*}
Using the classical estimate
\begin{equation}
\label{eq01016}
\sum_{d\leq y}\frac{\Lambda(d)}{d}
=
\log y+\Ocal(1),
\qquad y\geq2,
\end{equation}
we obtain
\begin{align*}
\sum_{\substack{d\ell\leq x\\(d,\ell)=1}}
\frac{\Lambda(d)}{d}
\frac{\tau(\ell)}{\ell}
&\ll
\sum_{\ell\leq x}
\frac{\tau(\ell)}{\ell}
\left(
\log\frac{x}{\ell}+1
\right)
\\
&\ll
\log x
\sum_{\ell\leq x}\frac{\tau(\ell)}{\ell}
\ll
(\log x)^3,
\end{align*}
where the last estimate follows from~\eqref{eqlem24}.

Applying~\eqref{Ex21} and~\eqref{eqlem24} to
\eqref{Gamma-tau-start}, we conclude that
\begin{multline*}
\sum_{n\leq x}\frac{1}{n}
\sum_{k\in\Reg_n}
\tau(\gcd(k,n))
\log\Gamma\!\left(\frac{k}{n}\right)
\\
=
\frac{\log\sqrt{2\pi}}{\zeta(2)}
\prod_p
\left(
1+
\frac{p(2p^2-1)}
{(p-1)^2(p+1)^3}
\right)x
+
\Ocal\!\left((\log x)^5\right).
\end{multline*}
This proves~\eqref{eqthm62}.\\

Finally, we establish~\eqref{eqthm63}. Taking $f=\mu$ in~\eqref{eq5},
we obtain
\begin{multline}
\label{Gamma-mu-start}
\sum_{n\leq x}\frac{1}{n}
\sum_{k\in\Reg_n}
\mu(\gcd(k,n))
\log\Gamma\!\left(\frac{k}{n}\right)
=
\log\sqrt{2\pi}
\left(
\sum_{\substack{d\ell\leq x\\(d,\ell)=1}}
\frac{\phi(d)}{d}
\frac{\mu(\ell)}{\ell}
-
\sum_{n\leq x}\frac{\mu(n)}{n}
\right)
\\
-
\frac12
\sum_{\substack{d\ell\leq x\\(d,\ell)=1}}
\frac{\Lambda(d)}{d}
\frac{\mu(\ell)}{\ell}.
\end{multline}

For the final sum, we take absolute values and use
$|\mu(\ell)|\leq1$. Thus,
\begin{align*}
\left|
\sum_{\substack{d\ell\leq x\\(d,\ell)=1}}
\frac{\Lambda(d)}{d}
\frac{\mu(\ell)}{\ell}
\right|
&\leq
\sum_{\ell\leq x}\frac{1}{\ell}
\sum_{d\leq x/\ell}\frac{\Lambda(d)}{d}
\\
&\ll
\sum_{\ell\leq x}
\frac{1}{\ell}
\left(
\log\frac{x}{\ell}+1
\right)
\\
&\ll
(\log x)^2.
\end{align*}
Using~\eqref{Ex31} and~\eqref{eqlem23} in
\eqref{Gamma-mu-start}, we therefore obtain
\begin{equation*}
\sum_{n\leq x}\frac{1}{n}
\sum_{k\in\Reg_n}
\mu(\gcd(k,n))
\log\Gamma\!\left(\frac{k}{n}\right)
=
\frac{\log\sqrt{2\pi}}{\zeta(2)}
\prod_p
\left(
1-\frac{1}{p(p+1)}
\right)x
+
\Ocal\!\left((\log x)^3\right).
\end{equation*}
This proves~\eqref{eqthm63}.\\

We next prove~\eqref{eqthm64}. Taking $f=\phi$ in~\eqref{eq6}, and using
the multiplicativity of $\phi$, we obtain
\begin{multline}
\label{eq0101111}
\sum_{n\leq x}\frac{1}{\phi(n)}
\sum_{k\in\Reg_n}
\phi(\gcd(k,n))
\log\Gamma\!\left(\frac{k}{n}\right)
=
\log\sqrt{2\pi}
\left(
\sum_{n\leq x}\tau^\star(n)
-
\sum_{n\leq x}1
\right)
-
\frac12
\sum_{\substack{d\ell\leq x\\(d,\ell)=1}}
\frac{\Lambda(d)}{\phi(d)}.
\end{multline}
Indeed, for $(d,\ell)=1$,
\[
\frac{\phi(\ell)\phi(d)}{\phi(d\ell)}=1,
\]
and hence the first double sum in~\eqref{eq6} counts the unitary divisors
of $n$.

It remains to evaluate
\[
S_\phi(x)
:=
\sum_{\substack{d\ell\leq x\\(d,\ell)=1}}
\frac{\Lambda(d)}{\phi(d)}.
\]
Reversing the order of summation gives
\[
S_\phi(x)
=
\sum_{d\leq x}
\frac{\Lambda(d)}{\phi(d)}
\sum_{\substack{\ell\leq x/d\\(\ell,d)=1}}1.
\]
Applying~\eqref{eqlem11} with $t=d$ and with $x$ replaced by $x/d$, we
obtain
\begin{equation}
\label{Sphi-decomposition}
S_\phi(x)
=
x\sum_{d\leq x}\frac{\Lambda(d)}{d^2}
-
\sum_{d\leq x}
\frac{\Lambda(d)}{\phi(d)}
\sum_{k\mid d}
\mu(k)
\theta\!\left(\frac{x}{kd}\right).
\end{equation}

Since $|\theta(y)|\leq 1/2$, we have
\begin{align*}
\left|
\sum_{d\leq x}
\frac{\Lambda(d)}{\phi(d)}
\sum_{k\mid d}
\mu(k)
\theta\!\left(\frac{x}{kd}\right)
\right|
&\ll
\sum_{d\leq x}
\frac{\Lambda(d)\tau^\star(d)}{\phi(d)}
\\
&\ll
\log x
\sum_{d\leq x}
\frac{\tau^2(d)}{d}
\\
&\ll
(\log x)^5.
\end{align*}
Here we used $\tau^\star(d)\leq\tau(d)$, the inequality $\phi(d)\tau(d)\geq d,$
and~\eqref{WS3}. Furthermore,
\[
-\frac{\zeta'(2)}{\zeta(2)}
=
\sum_{d\geq1}\frac{\Lambda(d)}{d^2},
\]
and
\[
\sum_{d>x}\frac{\Lambda(d)}{d^2}
\ll
\frac{\log x}{x}.
\]
Consequently,
\[
\sum_{d\leq x}\frac{\Lambda(d)}{d^2}
=
-\frac{\zeta'(2)}{\zeta(2)}
+
\Ocal\!\left(\frac{\log x}{x}\right).
\]
Substitution into~\eqref{Sphi-decomposition} yields
\begin{equation}
\label{Sphi-asymptotic}
S_\phi(x)
=
-\frac{\zeta'(2)}{\zeta(2)}x
+
\Ocal\!\left((\log x)^5\right).
\end{equation}

Using~\eqref{star}, the identity
\[
\sum_{n\leq x}1
=
x-\theta(x)-\frac12,
\]
and~\eqref{Sphi-asymptotic} in~\eqref{eq0101111}, we conclude that
\begin{multline*}
\sum_{n\leq x}\frac{1}{\phi(n)}
\sum_{k\in\Reg_n}
\phi(\gcd(k,n))
\log\Gamma\!\left(\frac{k}{n}\right)
=
\frac{\log\sqrt{2\pi}}{\zeta(2)}x\log x
\\
+
\frac{\log\sqrt{2\pi}}{\zeta(2)}
\left(
2\gamma-1
+\frac{\zeta'(2)}{\zeta(2)}
-\zeta(2)
+\frac{\zeta'(2)}
{2\log\sqrt{2\pi}}
\right)x
\\
+
\Ocal\!\left(
x^{1/2}
\exp\!\left(
-A(\log x)^{3/5}
(\log\log x)^{-1/5}
\right)
\right).
\end{multline*}
This proves~\eqref{eqthm64}.\\

Finally, we prove~\eqref{eqthm65}. Taking $f=\phi_2$ in~\eqref{eq6}, and
using the multiplicativity of $\phi_2$, we obtain
\begin{multline}
\label{eq01011111}
\sum_{n\leq x}\frac{1}{\phi_2(n)}
\sum_{k\in\Reg_n}
\phi_2(\gcd(k,n))
\log\Gamma\!\left(\frac{k}{n}\right)
\\
=
\log\sqrt{2\pi}
\left(
\sum_{\substack{d\ell\leq x\\(d,\ell)=1}}
\frac{\phi(d)}{\phi_2(d)}
-
\sum_{n\leq x}1
\right)
-
\frac12
\sum_{\substack{d\ell\leq x\\(d,\ell)=1}}
\frac{\Lambda(d)}{\phi_2(d)}.
\end{multline}

Set
\[
S_{\phi_2}(x)
:=
\sum_{\substack{d\ell\leq x\\(d,\ell)=1}}
\frac{\Lambda(d)}{\phi_2(d)}.
\]
Reversing the order of summation and applying~\eqref{eqlem11}, we obtain
\begin{equation}
\label{final}
S_{\phi_2}(x)
=
x\sum_{d\leq x}
\frac{\Lambda(d)\phi(d)}
{d^2\phi_2(d)}
-
\sum_{d\leq x}
\frac{\Lambda(d)}{\phi_2(d)}
\sum_{k\mid d}
\mu(k)
\theta\!\left(\frac{x}{kd}\right).
\end{equation}

Since
\[
\frac{\phi(d)}{\phi_2(d)}
=
\frac{1}{\psi(d)},
\]
the first sum in~\eqref{final} may be written as
\[
\sum_{d\leq x}
\frac{\Lambda(d)}
{d^2\psi(d)}
=
K-
\sum_{d>x}
\frac{\Lambda(d)}
{d^2\psi(d)},
\]
where
\[
K
:=
\sum_{d\geq1}
\frac{\Lambda(d)}
{d^2\psi(d)}.
\]
Since $\psi(d)\geq d$, we have
\[
\sum_{d>x}
\frac{\Lambda(d)}
{d^2\psi(d)}
\ll
\sum_{d>x}\frac{\log d}{d^3}
\ll
\frac{\log x}{x^2}.
\]
Therefore,
\begin{equation}
\label{K-main-term}
x\sum_{d\leq x}
\frac{\Lambda(d)\phi(d)}
{d^2\phi_2(d)}
=
Kx
+
\Ocal\!\left(\frac{\log x}{x}\right).
\end{equation}

For the second term in~\eqref{final}, we use
$|\theta(y)|\leq1/2$ and $\tau^\star(d)\leq\tau(d)$ to obtain
\begin{align*}
\left|
\sum_{d\leq x}
\frac{\Lambda(d)}{\phi_2(d)}
\sum_{k\mid d}
\mu(k)
\theta\!\left(\frac{x}{kd}\right)
\right|
&\ll
\sum_{d\leq x}
\frac{\Lambda(d)\tau(d)}
{\phi_2(d)}
\\
&\ll
\log x
\sum_{d\leq x}
\frac{\tau^2(d)}{d^2}
\\
&\ll
\log x.
\end{align*}
Here we used
\[
\phi_2(d)\geq\frac{d^2}{\tau(d)}
\]
and the convergence of
\[
\sum_{d\geq1}\frac{\tau^2(d)}{d^2}.
\]
Combining this estimate with~\eqref{final} and~\eqref{K-main-term}, we get
\begin{equation}
\label{Sphi2-asymptotic}
S_{\phi_2}(x)
=
Kx+\Ocal(\log x).
\end{equation}

Finally, applying~\eqref{Ex41},
\[
\sum_{n\leq x}1
=
x-\theta(x)-\frac12,
\]
and~\eqref{Sphi2-asymptotic} to~\eqref{eq01011111}, we obtain
\begin{multline*}
\sum_{n\leq x}\frac{1}{\phi_2(n)}
\sum_{k\in\Reg_n}
\phi_2(\gcd(k,n))
\log\Gamma\!\left(\frac{k}{n}\right)
\\
=
\log\sqrt{2\pi}
\left(
\prod_p
\left(
1+\frac{1}{(p+1)^2}
\right)
-
1
-
\frac{K}{2\log\sqrt{2\pi}}
\right)x
+
\Ocal\!\left((\log x)^4\right).
\end{multline*}
This proves~\eqref{eqthm65} and completes the proof of
Theorem~\ref{Thm6}.

\section*{Acknowledgement}
The authors sincerely thank Isao Kiuchi and Kohji Matsumoto for their careful reading of the manuscript and for their valuable comments and suggestions.

\medskip\noindent {Waseem Alass: 
Johannes Kepler University Linz, Altenbergerstrasse 69, 4040 Linz, Austria. E-mail: {\tt waseem.alass@jku.at}}

\medskip\noindent {Sumaia Saad Eddin: 
Johann Radon Institute for Computational and Applied Mathematics, Austrian Academy of Sciences, Altenbergerstrasse 69, 4040 Linz, Austria.}\\
E-mail: {\tt sumaia.saad-eddin@ricam.oeaw.ac.at}}


\begin{thebibliography}{99}

\bibitem{WA} W. Alass and S. Saad Eddin, Weighted averages of arithmetic functions over regular integers modulo $n$, 	arXiv:1712.05503.

\bibitem{A.T} B. Apostol and L. T\'{o}th, Some remarks on regular integers modulo $n$, {\it Filomat} \textbf{29} (2015), 687-701.

\bibitem{H} M. N. Huxley, Exponential sums and lattice points III, 
{\it Proc. London Math. Soc.} {\bf 87} (2003),  591-609.

\bibitem{Jia} R.Q. Jia, Estimation of partial sums of series $\sum\mu(n)/n$, {\it Kexue Tongbao} \textbf{30} (1985), 575-578.

\bibitem{K2}  I. Kiuchi, Sums of averages of $\gcd$-sum functions,
{\it Journal of Number Theory} \textbf{176} (2017), 449-472.

\bibitem{K.M}  I. Kiuchi and K. Matsuoka, Remarks on a paper by B. Apostol and L. T\'{o}th,
{\it J. Ramanujan Math. Soc.} \textbf{34} (2019), 43-57.

\bibitem{Pillai}  S.S. Pillai, On an arithmetic function,
{\it J. Annamalai Univ.} \textbf{2} (1933), 243-248.

\bibitem{SMC}  J. S\'{a}ndor, D.S. Mitrinovi\'{c} and B. Crstici, {\it Handbook of Number Theory I}, Springer, (2006).

\bibitem{Sur} D. Suryanarayana, The greatest divisor of $n$ which is prime to $k$, {\it Math. Student} \textbf{37} (1969), 147-157.

\bibitem{Sur1} D. Suryanarayana and V.S.R. Prasad, The number of $k$-free divisors of an integer, {\it Acta. Arith.} \textbf{37} (1970/71), 345-354.

\bibitem{T2} L. T\'{o}th, A $\gcd$-sum function over regular integers modulo $n$, {\it J. Integer Sequences} \textbf{12} (2009), Article 09.2.5.

\bibitem{T1} L. T\'{o}th, A survey of $\gcd$-sum functions, {\it J. Integer Sequences} \textbf{13} (2010), Article 10.8.1.

\bibitem{W1} B.M. Wilson, Proofs of some formulae enunciated by Ramanujan, {\it Proc. London Math. Soc.} \textbf{21} (1922), 235-255.


\end{thebibliography}
\end{document}